\documentclass[11pt,a4paper]{article}
\usepackage[T1]{fontenc}
\usepackage{lmodern}
\usepackage[margin=25mm]{geometry}
\usepackage{amsmath,amssymb,amsthm,booktabs,graphicx,microtype,enumitem,caption,titling,float,xcolor}
\usepackage[colorlinks=true,linkcolor=black,citecolor=black,urlcolor=black]{hyperref}
\setlist{itemsep=3pt,topsep=4pt,leftmargin=*}
\newtheorem{theorem}{Theorem}
\newtheorem{lemma}{Lemma}
\newtheorem{proposition}{Proposition}
\newtheorem{corollary}{Corollary}
\newcommand{\lo}{\lambda}
\newcommand{\xor}{\mathbin{\oplus}}

\title{Constructing longer snakes and\\improved asymptotic bounds in hypercubes}
\author{Tom Taylor}
\date{}
\hypersetup{pdftitle={Constructing longer snakes and improved asymptotic bounds in hypercubes},pdfauthor={Tom Taylor}}

\begin{document}
\maketitle
\begin{abstract}
% Source Tab 5 block 4
We give snakes that are longer than the previous best-known in dimensions 13--20 and improve the general lower bound for every dimension $d\ge21$. Our explicit snakes reach 371,711 edges in dimension 20. Twenty compatible paths in that cube allow generalisation to give snakes of length at least $(17/48)2^d$ for every $d\ge21$. Their controlled overlaps allow copies to be joined across the layers of a larger cube without creating shortcuts. Four additional paths give the same bound for coils. We explain the construction, prove the joining rule, then count its length. The finite paths and their required intersections are independently verifiable.

\end{abstract}

\section{The problem and earlier work}
% Source Tab 5 block 6
The hypercube $Q_d$ has $2^d$ vertices, represented by binary strings of length $d$. Two vertices are joined by an edge if and only if they differ in exactly one coordinate. Therefore, choosing a set of vertices determines all the edges between them. A snake is a set whose induced graph is a single path. Equivalently, its vertices can be ordered without repetition so that consecutive vertices are adjacent and no other pair is adjacent. An extra edge between nonconsecutive vertices is called a chord, or shortcut. For example, $000,001,011$ form a snake, adding $010$ closes a square.

% Source Tab 5 block 7
We measure length in edges, a snake with $L$ edges contains $L+1$ vertices. Write $a(d)$ for the greatest possible snake length in $Q_d$. Constructing and checking one snake of length $L$ proves the lower bound $a(d)\ge L$.

% Source Tab 5 block 8
Earlier work combines searches for long paths in individual dimensions with constructions that apply in every sufficiently large dimension. Abbott and Katchalski obtained the leading coefficient $77/256$ [1, 2]. Evdokimov describes how compatible paths can be joined in Gray-ordered layers [5; pp. 58–59]. Computational work includes the record constructions of Allison and Paulusma, the recent searches of Orland et al. and Bernatonis, and Itty’s posted paths in dimensions 12 and 13 [3, 7, 8, 4]. We combine finite search with the layer construction, using longer compatible paths to improve the resulting bound.

% Source Tab 5 block 9
First we extend the existing ‘record snakes’ by a combination of local iterations and repairs. Then we find a specific set of snakes in $Q_{20}$, with precise overlap, which allow construction of longer snakes in all higher dimensions.

\section{The new bounds}
% Source Tab 5 block 11
Table 1 displays our verified finite snakes. The construction in Section 4 establishes the following bound in every higher dimension.

\begin{theorem}\label{thm:main}
% Source Tab 5 block 12
For every integer $d\ge21$,

\[
a(d)\ \ge\ \left\lceil\frac{928503}{2621440}\,2^d+\frac{683}{5}\right\rceil\ \ge\ \frac{17}{48}\,2^d.
\]

\end{theorem}

% Source Tab 5 block 14
The clean coefficient $17/48\approx0.35417$ improves the classical Abbott–Katchalski coefficient $77/256\approx0.30078$ by about $17.75\%$ [1, 2].

\begin{table}[H]
\centering\small
\begin{tabular}{rrrr}
\toprule
Dimension & Earlier snake comparison & Verified snake & Growth\\
\midrule
13 & 2,934 & 2,938 & $0.14\%$\\
14 & 5,750 & 5,844 & $1.63\%$\\
15 & 11,249 & 11,673 & $3.77\%$\\
16 & 21,411 & 23,318 & $8.91\%$\\
17 & 40,835 & 46,576 & $14.06\%$\\
18 & 78,958 & 92,989 & $17.77\%$\\
19 & 157,898 & 185,903 & $17.74\%$\\
20 & 315,798 & 371,711 & $17.71\%$\\
\bottomrule
\end{tabular}
\caption{% Source Tab 5 block 16
Snake lengths in edges, with percentage growth over the listed comparisons. Sources are Itty for $d=13$ [4], Bernatonis for $d=14$–17 [8], and the Abbott–Katchalski values tabulated by Allison and Paulusma for $d=18$–20 [3].}
\end{table}

% Source Tab 5 block 17
A coil is a cycle with no chord. Removing one vertex from a coil of length $C$ gives a snake of length $C-2$. Whilst a coil necessarily creates snakes, snakes do not necessarily produce a coil. A fresh closure search starting from the snakes in Table 1 gives the coils in Table 2. The general construction also gives $b(d)\ge(17/48)2^d$ for the greatest coil length $b(d)$ in every $d\ge21$, as Section 4.6 proves.

\begin{table}[H]
\centering\small
\begin{tabular}{rrrr}
\toprule
Dimension & Earlier coil comparison & Verified coil & Growth\\
\midrule
14 & 4,934 & 5,750 & $16.54\%$\\
15 & 9,868 & 11,566 & $17.21\%$\\
16 & 19,740 & 23,206 & $17.56\%$\\
17 & 39,480 & 46,470 & $17.71\%$\\
18 & 78,960 & 92,850 & $17.59\%$\\
19 & 157,900 & 185,778 & $17.66\%$\\
20 & 315,800 & 371,586 & $17.66\%$\\
\bottomrule
\end{tabular}
\caption{% Source Tab 5 block 19
Coils obtained from the latest snakes. Comparisons are from Allison and Paulusma [3].}
\end{table}

\section{Finding the finite paths}
% Source Tab 5 block 21
A lift joins two paths in dimension $x$ to make a snake in dimension $x+1$. Give vertices of the first path a new coordinate $0$, and those of the second a new coordinate $1$ such that $v$ becomes $(v,0)$ or $(v,1)$. Each path keeps its existing edges. Between the layers, $(v,0)$ and $(u,1)$ are adjacent if and only if $v=u$. Thus every shared vertex of the original paths creates a cross-layer edge.

% Source Tab 5 block 22
For this simple lift to form one snake, the original paths must share exactly one vertex, which is an endpoint of both. It supplies the single joining edge, giving length $L_1+L_2+1$ from paths of lengths $L_1$ and $L_2$. Our search seeks two long paths with this property.

% BEGIN SEARCH BOX
\begin{center}
\fcolorbox{black!45}{white}{\begin{minipage}{0.94\linewidth}
\small
\textbf{Iterative search}\par\smallskip
\textbf{Input:} a verified snake $A$, a higher target dimension $D$, and a computation limit.\par
\textbf{While $A$ has fewer than $D$ dimensions and compute remains:}\par
1. Choose a coordinate $c$ that $A$ changes rarely. Flip it at every vertex to form $B$.\par
2. Choose an endpoint $a$ of $A$ as the joining vertex.\par
3. Repair $B$ using shortest bridges through permitted vertices. Require $B$ to be induced, to end at $a$, and to share only $a$ with $A$. If the repair fails these checks, retry the loop with different choices.\par
4. Put $A$ and $B$ in the two layers of a cube with one extra coordinate. Join their copies of $a$, and call the resulting snake $A$ (to iterate).\par
5. Improve $A$ while keeping its endpoints fixed and exactly one crossing of the new coordinate.\par\smallskip
\textbf{After the lifts:} use remaining compute for endpoint pivots and local improvements.\par
Return the longest verified snake in the highest dimension reached.
\end{minipage}}
\end{center}
% END SEARCH BOX
Flipping a coordinate used $q$ times gives exactly $2q$ shared vertices.
Permitted bridges avoid the vertices of $A$ other than the chosen endpoint $a$, and unwanted contacts with the retained parts of $B$.
A pivot preserves length but can expose a new extension;
Appendix~\ref{app:pivot} gives the precise move.\footnote{We also extract lower-dimensional snakes from uninterrupted sections in which one coordinate stays fixed. Deleting that coordinate preserves the section’s length and all its adjacencies, allowing improvements found in higher dimensions to feed back into lower ones. We call this face projection.}

% Source Tab 5 block 28
The $Q_{12}$ and $Q_{13}$ inputs, of lengths 1,480 and 2,934, are Nathaniel Itty’s publicly posted paths [4]. Our higher-dimensional chain starts from Bernatonis’s paths, extending an Orland et al. seed [8, 7]. Local rerouting and joining lower-dimensional snakes have precedents in Brown, Wynn and Echols [9, 10, 11].

% Source Tab 5 block 29
For coils, we try short closing paths between the snake’s endpoints, sometimes trimming a short section first.

% Source Tab 5 block 30
To generalise beyond the dimensions searched, we need a collection of paths whose overlaps remain controlled when many copies are used. We searched for compatible paths in $Q_{20}$ and improved them while preserving their required intersections. The resulting collection is the finite input to the proof below.

\section{Generalising the construction to higher dimensions}
% Source Tab 5 block 32
We now use twenty specific paths found in $Q_{20}$ to construct snakes in every larger dimension. Write $d=20+m$ with $m\ge1$. Each vertex of $Q_d$ is a pair $(v,b)$, where $v\in Q_{20}$ and $b\in Q_m$. Fixing the last $m$ coordinates gives one layer, a copy of $Q_{20}$, so there are $2^m$ layers. We choose one of our twenty paths for each layer and join the copies into a single snake. The same path can be used in many layers (if they are chosen correctly).

% Source Tab 5 block 33
Two neighbouring layers have an edge between their chosen paths wherever those paths share an inner vertex $v$. We therefore need paths whose overlaps occur only at controlled joining points. Once those paths have been checked, the rule below works in every dimension without further path search.

\subsection{The fixed collection of paths}
% Source Tab 5 block 35
Our twenty paths use seven (distinct) particular vertices of $Q_{20}$ as endpoints: two roots, $r_0,r_1$, and five tips, $t_0,\ldots,t_4$ (as in Evdokimov’s construction [5; pp. 58–59]). These are ordinary cube vertices; the names distinguish their roles in the joining rule. For each root-tip pair we have two different induced paths, one in each of two families. We write them in the directions in which they will be traversed:

\[
I_r(t):t\longrightarrow r,\qquad O_r(t):r\longrightarrow t.
\]

% Source Tab 5 block 37
Thus there are $2\times5\times2=20$ paths. The two families are different routes, not the same route read backwards. Each $I$ path and each $O$ path share only the named endpoints they have in common. For example, $I_{r_0}(t_0)$ and $O_{r_0}(t_1)$ meet only at $r_0$, so putting them in neighbouring layers creates exactly the desired joining edge. An $I$ path and an $O$ path with different roots and different tips are disjoint, so their copies have no cross-layer edge.

% Source Tab 5 block 38
Precisely, writing $V(R)$ for the set of vertices in a path $R$, our finite paths satisfy

\begin{equation}\label{eq:intersection}
 V(I_r(t))\cap V(O_s(t'))
 =\bigl(\{r\}\text{ if }r=s\bigr)\ \cup\
   \bigl(\{t\}\text{ if }t=t'\bigr).
\end{equation}

% Source Tab 5 block 40
If both endpoints agree, the intersection consists of those two endpoints. Routes in the same family may overlap extensively: we will put them in layers that cannot be neighbours.

% Source Tab 5 block 41
Table 3 gives the vertex counts of the supplied paths. Their average is 371,401.2, about $35.42\%$ of the $2^{20}$ vertices in a layer. Using paths of roughly this size throughout the larger cube suggests the same occupied fraction there. We next prove that the copies can be joined safely.

\begin{table}[H]
\centering\small
\begin{tabular}{rrrrr}
\toprule
Tip & $|I_{r_0}|$ & $|I_{r_1}|$ & $|O_{r_0}|$ & $|O_{r_1}|$\\
\midrule
$t_0$ & 371,524 & 371,411 & 371,518 & 371,431\\
$t_1$ & 371,540 & 371,427 & 371,466 & 371,379\\
$t_2$ & 371,483 & 371,370 & 371,485 & 371,398\\
$t_3$ & 371,429 & 371,316 & 371,447 & 371,360\\
$t_4$ & 371,396 & 371,283 & 371,224 & 371,137\\
\bottomrule
\end{tabular}
\caption{% Source Tab 5 block 43
The twenty certified route sizes.}
\end{table}

\subsection{Joining copies across the layers}
% Source Tab 5 block 45
First consider $Q_{21}$, which has just two layers. Put $I_r(t_0)$ in one and $O_r(t_0)$ in the other. They share the root $r$ and tip $t_0$, so their copies have two cross-layer edges. Delete $t_0$ from both paths. The remaining paths meet across the layers only at $r$ and form a snake. If their original sizes sum to $V$, the snake has $V-2$ vertices and $V-3$ edges.

% Source Tab 5 block 46
In general, write $N=2^{m-1}$ and visit them in binary reflected Gray order,

\begin{equation}\label{eq:grayorder}
g(i)=i\xor\lfloor i/2\rfloor,\qquad 0\le i<2^m=2N,
\end{equation}

% Source Tab 5 block 48
where $\oplus$ is bitwise exclusive-or. For example, with three outer bits the order is

\[
000,001,011,010,110,111,101,100.
\]

% Source Tab 5 block 50
Each step changes one bit, so successive layers are neighbours. There are also neighbouring layers that are not consecutive in this order, and we must prevent unwanted crossings between them.

% Source Tab 5 block 51
Put an $I$ route in every even-indexed layer and an $O$ route in every odd-indexed layer. Every neighbouring pair of layers therefore contains opposite families, whose shared inner vertices are exactly the roots and tips specified by (1).

\textbf{The roots.} % Source Tab 5 block 52
Join layers at positions 0 and 1 at $r_0$, positions 2 and 3 at $r_1$, positions 4 and 5 at $r_0$, and continue alternating. This rule has a useful description directly in the outer bits: use $r_0$ when the number of ones excluding the rightmost bit is even, and $r_1$ when it is odd. Indeed, deleting the rightmost bit of $g(i)$ gives $g(\lfloor i/2\rfloor)$, whose parity is $\lfloor i/2\rfloor\bmod2$.

% Source Tab 5 block 53
Flipping the rightmost outer bit preserves the root, and these edges are exactly the intended root joins. Flipping any other outer bit changes the root. Thus two neighbouring layers share a root only when we mean to join them there.

\textbf{The tips.} % Source Tab 5 block 54
The remaining joins connect positions 1 and 2, positions 3 and 4, and so on. Call these tip pairs 1, 2, …, $N-1$. Pair $x$ consists of layers $g(2x-1)$ and $g(2x)$, which must both use the same tip. Write $c(x)\in\{0,\ldots,4\}$ for the index of this tip, so pair $x$ uses $t_{c(x)}$. The complete sequence of joins is

\[
r_0,\ t_{c(1)},\ r_1,\ t_{c(2)},\ r_0,\ t_{c(3)},\ldots.
\]

% Source Tab 5 block 56
The first and last layers have unused boundary tips. We choose $t_0$ for both and delete those two tip vertices, as in the two-layer example.

\begin{figure}[H]
\centering
\includegraphics[width=\textwidth]{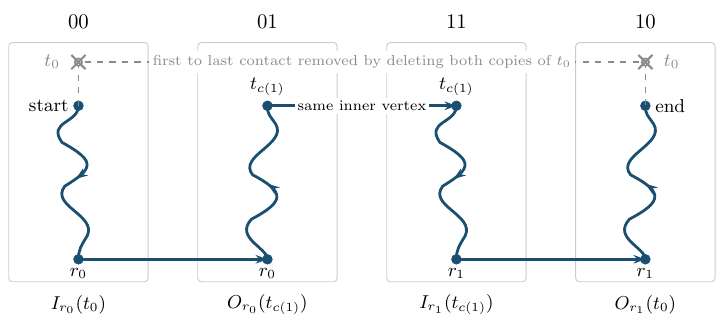}
\caption{% Source Tab 5 block 58
Four layers, joined at $r_0$, $t_{c(1)}$ and $r_1$. Each box is a whole $Q_{20}$ layer and each curve a long route. Deleting the two copies of $t_0$ gives a snake.}
\end{figure}

\subsection{Which pairs can reuse a tip}
% Source Tab 5 block 60
If we reuse a tip in another pair, neither layer of that pair may neighbour either layer of the earlier pair. Otherwise the two copies of the tip have identical inner bits and differ in just one outer bit, creating an extra crossing. This restriction is captured by a graph $H$:

\begin{enumerate}
\item % Source Tab 5 block 61
Draw the outer cube, with a vertex for each layer and every edge between neighbouring layers, including edges not used by the Gray traversal.

\item % Source Tab 5 block 62
Replace each intended tip pair by a single vertex, labelled by its position $x$. Omit the first and last layers from this auxiliary graph because their tips are deleted.

\item % Source Tab 5 block 63
Join two different vertices $x,y$ whenever any layer in pair $x$ neighbours any layer in pair $y$. Several contacts between two pairs yield just one edge of $H$.

\end{enumerate}
% Source Tab 5 block 64
Thus an edge of $H$ means that its two tip pairs must receive different tips. This includes contacts at intended root joins (reusing a tip there would create a second crossing alongside the root crossing).

\begin{samepage}
% Source Tab 5 block 65
Choosing tips now means splitting the vertices of $H$ into five groups, one per available tip, with no edge inside any group. Such a group is an independent set (the partition is a proper colouring).

\end{samepage}
% Source Tab 5 block 66
For example, when $m=4$ there are sixteen layers and seven tip pairs. A valid assignment of the five tips is

\[
\begin{array}{c|ccccc}\text{tip}&t_0&t_1&t_2&t_3&t_4\\\hline\text{pair positions}&1&2,5&4,7&3,6&\text{none}\end{array}
\]

% Source Tab 5 block 68
To see why pairs 2 and 5 can share tip $t_1$, their outer labels are

\[
\begin{array}{ccl}\text{pair 2:}&0010\;\longleftrightarrow\;0110,\\\text{pair 5:}&1101\;\longleftrightarrow\;1111.\end{array}
\]

% Source Tab 5 block 70
Within each pair one bit changes and between pairs every comparison differs in at least two bits. The desired crossings exist, with no additional crossing caused by reusing $t_1$.

We choose these paths as follows:
\begin{equation}\label{eq:layers}
\begin{array}{c|l}
\text{layer}&\text{path}\\\hline
 g(0)&I_{r_0}(t_0)\\
 g(2x-1)&O_{r_{(x-1)\bmod2}}(t_{c(x)})\\
 g(2x)&I_{r_{x\bmod2}}(t_{c(x)})\\
 g(2N-1)&O_{r_{(N-1)\bmod2}}(t_0).
\end{array}
\end{equation}
The middle two rows apply for $1\le x<N$; when $N=1$ they are absent.
Delete the two boundary copies of $t_0$ and join successive paths at their common terminals by changing one outer bit.

\begin{lemma}\label{lem:assembly}
% Source Tab 5 block 75
A proper colouring of $H$ makes the assembled path induced.

\end{lemma}

\begin{proof}
% Source Tab 5 block 76
The routes are induced within their layers. Between neighbouring layers, the intersection rule (1) permits contacts only at a shared root or tip. The root rule permits exactly the intended root joins. Within each tip pair, the roots differ and the shared tip supplies exactly one crossing. Between different tip pairs, an outer adjacency gives an edge of $H$, so their tips differ and create no crossing. Boundary tips have been deleted, leaving only the intended root contact at each boundary layer. These cases account for every possible edge, so the joined path is induced.

\end{proof}

% Source Tab 5 block 77
Note. If $V$ is the sum of the untrimmed route sizes then the path has $V-3$ edges.

\subsection{Why five tips suffice}
% Source Tab 5 block 79
Our construction uses the colouring approach described by Evdokimov [5], based on his earlier joint work with Glagolev [6]. We give a self-contained colouring argument for the precise conflict graph required here.

% Source Tab 5 block 80
We first express the geometric definition of $H$ as an arithmetic test. For a positive integer $n$, let $k$ be the largest nonnegative integer for which $2^k$ divides $n$, and set $\lambda(n)=2^k$. Thus $\lambda(n)$ is the largest power of two dividing $n$; for example, $\lambda(12)=4$ and $\lambda(13)=1$. In standard notation it is $2^{\nu_2(n)}$. The test is

\begin{equation}\label{eq:H}
 x\sim y\quad\Longleftrightarrow\quad x\ne y\ \text{ and }\ |x-y|\le\lo(x+y).
\end{equation}

% Source Tab 5 block 82
Here $x,y$ are tip-pair positions, not outer bit strings or inner tip vertices.

% Source Tab 5 block 83
To derive the test, invert the Gray labelling by taking cumulative exclusive-or from the highest bit downwards. Changing one Gray bit complements a suffix of the binary index. Hence neighbouring layers have indices with the same prefix and complementary final $k$ bits, for some $k\ge1$. They lie in the same aligned block of $2^k$ integers; their sum plus one is an odd multiple of $2^k$, and their difference is less than $2^k$. Conversely, those two conditions put the indices on opposite sides of the block’s midpoint with complementary suffixes. Therefore

\begin{equation}\label{eq:gray}
 g(i)\sim g(j)\quad\Longleftrightarrow\quad |i-j|<\lo(i+j+1)
 \qquad(i\ne j).
\end{equation}

% Source Tab 5 block 85
For distinct tip pairs $x,y$, only opposite-parity layer indices can be neighbours. Substituting $(i,j)=(2x-1,2y)$ or $(2x,2y-1)$ gives

\[
|2(x-y)-1|<2\lo(x+y)\quad\text{or}\quad|2(x-y)+1|<2\lo(x+y).
\]

% Source Tab 5 block 87
Since $x-y$ is an integer, their union is exactly (4).

% Source Tab 5 block 88
We now colour this graph with five labels. Assigning a different actual tip to each label gives the required tip groups.

\begin{lemma}\label{lem:colours}
% Source Tab 5 block 89
The graph $H$ has an explicit proper colouring with five labels.

\end{lemma}

The proof is in Appendix~\ref{app:colouring}.

\subsection{Counting the length}
% Source Tab 5 block 101
We may change which actual tip is assigned to each colour, provided we change the assignment throughout the construction. This preserves validity however can change the length, because the twenty routes have different sizes. We compare the five cyclic assignments of tips to colours, each with both root orders: ten complete valid constructions. At least one is as long as their average. These choices cyclically shift $c(x)$ modulo five and exchange the roots; the boundary tip stays fixed at $t_0$.

% Replace the instruction at block 102, preserving the sentence in block 103.
For $0\le j<5$, define $S_j$ as the sum of the four path sizes for tip $t_j$:
\[
 S_j:=|I_{r_0}(t_j)|+|I_{r_1}(t_j)|
     +|O_{r_0}(t_j)|+|O_{r_1}(t_j)|.
\]
The total of all twenty path sizes is $T:=\sum_{j=0}^4 S_j$ ($T=7{,}428{,}024$, $S_0=1{,}485{,}884$).
Let $L$ be the snake’s length in edges. Count them from start to finish. The first path loses its boundary tip, leaving $|I_{r_0}(t_0)|-2$ edges. Entering each interior path adds one crossing edge and its internal edges, hence adds its vertex count. The last path also loses its boundary tip, so entering and traversing it adds its vertex count minus one. Thus
\begin{equation}\label{eq:counttotal}
\begin{aligned}
 L={}&|I_{r_0}(t_0)|-2\\
 &+\sum_{x=1}^{N-1}\bigl(|O_{r_{(x-1)\bmod2}}(t_{c(x)})|
                      +|I_{r_{x\bmod2}}(t_{c(x)})|\bigr)\\
 &+|O_{r_{(N-1)\bmod2}}(t_0)|-1.
\end{aligned}
\end{equation}
Now average over the ten choices, with lengths $L_1,\ldots,L_{10}$. The first term has mean
\[
 \frac{|I_{r_0}(t_0)|+|I_{r_1}(t_0)|}{2}-2.
\]
At each interior tip pair, cycling the tips and swapping the roots uses each of the twenty paths once, giving mean $T/10$. Finally, the last term has mean
\[
 \frac{|O_{r_0}(t_0)|+|O_{r_1}(t_0)|}{2}-1.
\]
The first and last means sum to $S_0/2-3$. Consequently
\begin{equation}\label{eq:countmean}
 \max_q L_q\ge
 (N-1)\frac{T}{10}+\frac{S_0}{2}-3
 =2^m\frac{1857006}{5}+\frac{683}{5}.
\end{equation}
Substituting $m=d-20$ gives
\[
\begin{aligned}
\max_q L_q
&\ge 2^{d-20}\frac{1857006}{5}+\frac{683}{5}\\
&=\frac{928503}{2621440}\,2^d+\frac{683}{5}
>\frac{17}{48}\,2^d.
\end{aligned}
\]
Since the length is an integer, rounding up proves Theorem~1.

\subsection{Closing the construction into coils}
% Revised from Tab 7 instructions.
To extend our snake construction to coils, we reserve a sixth tip, $t_5$, solely for joining the ends. The supplement supplies four verified paths to $t_5$ in $Q_{20}$, satisfying (1) simultaneously with the twenty original paths. Their sizes, in the column order of Table 3, are 371,397, 371,284, 371,067 and 370,980 vertices. Let $S_5$ be their total size ($S_5=1,484,728$).

\begin{corollary}\label{cor:coils}
% Source Tab 5 block 116
For every integer $d\ge21$, $b(d)\ge(17/48)2^d$.

\end{corollary}

The proof is in Appendix~\ref{app:coil-proof}.

\section{Verification and the remaining gap}
% Source Tab 5 block 124
The public repository (https://github.com/tommyet/snakes) supplies the twenty snake paths and four additional coil paths, expandable to vertex lists, with a verifier.

% Source Tab 5 block 125
Huang proved that more than half the vertices of $Q_d$ always induce a vertex of degree at least $\sqrt d$ [13]. A snake has a maximum degree of two. Hence $a(d)\le2^{d-1}-1$ for $d\ge5$, and in particular

\[
371,711\ \le\ a(20)\ \le\ 524,287.
\]

% Source Tab 5 block 127
Whilst our snake reaches $70.90\%$ of this upper bound, the optimum remains unknown. I expect short-timelines to lead to frequent finding of ‘record snakes’. Longer compatible routes in $Q_{20}$ would improve the general bound without changing the joining proof. The leading coefficient is the mean size of the twenty routes divided by $2^{20}$.

% Source Tab 5 block 128
\paragraph{AI assistance.}
An LLM coding assistant (OpenAI’s Astra) supported the search and programming under the author’s direction.

\clearpage
\section*{Appendices}
\appendix
\section{Endpoint pivots}
\label{app:pivot}
\begin{proposition}
% Source Tab 5 block 145
Let $P=(p_0,\ldots,p_L)$ be induced. Suppose an unused vertex $v$ has exactly two neighbours in $V(P)$, namely $p_j,p_L$, with $0\le j\le L-2$. Then

\[
(p_0,\ldots,p_j,v,p_L,p_{L-1},\ldots,p_{j+2})
\]

% Source Tab 5 block 147
is an induced path of the same length $L$.

\end{proposition}

\begin{proof}
% Source Tab 5 block 148
Delete $p_{j+1}$, add $v$ and reverse the retained suffix. Consecutive vertices remain adjacent. The retained old vertices have no extra edges, and $v$ has only its two prescribed neighbours. The vertex count is unchanged.

\end{proof}

% Source Tab 5 block 149
The new endpoint is $p_{j+2}$. A previously unused vertex with exactly one neighbour on the new path, at that endpoint, can now extend it. The pivot itself adds nothing; it changes which extension is possible.

\begin{figure}[H]
\centering
\includegraphics[width=\textwidth]{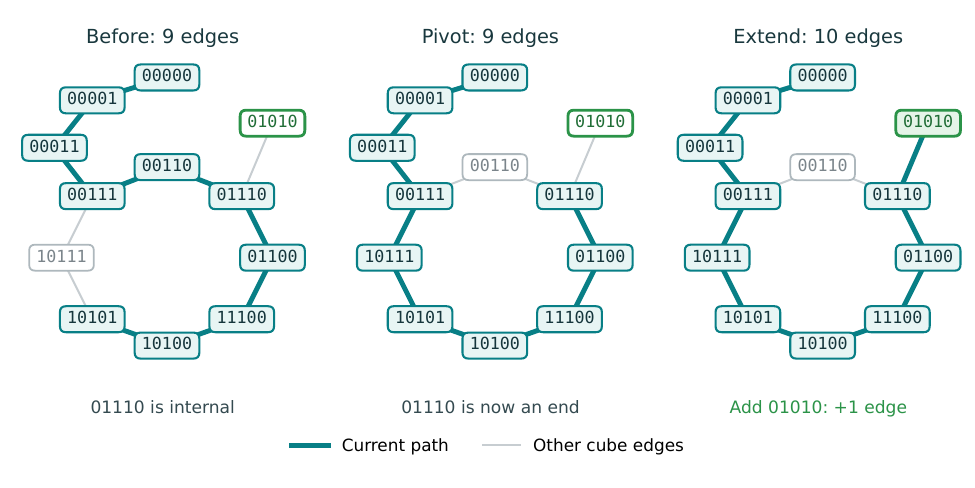}
\caption{% Source Tab 5 block 151
An endpoint pivot in $Q_5$. Every label gives the vertex’s five binary coordinates. Replacing $00110$ by $10111$ preserves nine edges and makes $01110$ an endpoint, allowing $01010$ to extend the snake to ten edges. Neither original endpoint has a legal extension anywhere in $Q_5$.}
\end{figure}

\clearpage
\section{Five-tip colouring}
\label{app:colouring}

\begin{proof}
Recall that distinct positions $x,y$ are joined in $H$ exactly when
\[
|x-y|\le\lambda(x+y),
\]
where $\lambda(n)$ is the largest power of two dividing $n$. We will use the identity
\[
\lambda(Q\pm r)=\lambda(r)\qquad\text{for }0<r<Q,
\]
whenever $Q$ is a power of two.

We construct the colouring by induction, extending a block of positions $1,\ldots,M-1$ to $1,\ldots,2M-1$, where $M$ is a power of two.

The new middle vertex will have neighbours precisely at power-of-two distances. We therefore maintain an additional condition on the colours at these boundary positions, ensuring that a colour remains available for the middle vertex.

Alongside a proper colouring, we maintain the following condition: the positions at distances $1,2,4,\ldots$ from the two boundaries use two permitted palettes of three colours, sharing exactly one colour. For the current step, name these palettes so that
\[
\begin{array}{c|c}
\text{positions}&\text{permitted colours}\\\hline
1,2,4,\ldots,M/2&A,B,C\\
M-1,M-2,M-4,\ldots,M/2&A,D,E.
\end{array}
\]
Not every permitted colour needs to occur and the boundaries $0,M$ themselves are outside the block.

For $M=2$, colour the single position $1$ with $A$. Both conditions trivially hold.

For the induction step, keep the old block on the left and reflect a copy across the new middle position $M$. Thus position $x$ is copied to $2M-x$. Give the middle colour $B$, and recolour the reflected copy according to
\[
\begin{array}{c|ccccc}
\text{old colour}&A&B&C&D&E\\\hline
\text{new colour}&D&E&B&A&C.
\end{array}
\]

We check the three possible types of edge:
\begin{enumerate}
\item \textbf{Edges within a half.} The left half is unchanged. Reflection preserves the conflict test: it leaves $|x-y|$ unchanged and replaces $x+y$ by $4M-(x+y)$, which has the same value of $\lambda$. The reflected half therefore has exactly the old conflict pattern. Since the recolouring is a permutation of the colours, it remains properly coloured.

\item \textbf{Edges between the halves.} Write the two positions as $M-a$ and $M+b$, with $1\le a,b<M$. If $a\ne b$, then
\[
\lambda(2M+b-a)=\lambda(|b-a|)
\le |b-a|<a+b,
\]
so they are not joined. If $a=b$, their sum is $2M$, and their distance $2a$ is less than $\lambda(2M)=2M$, so they are joined. Thus only mirror positions conflict. Every colour changes under our recolouring, so every mirror pair receives different colours.

\item \textbf{Edges at the middle.} A position $M\pm a$, with $1\le a<M$, is joined to $M$ exactly when
\[
a\le\lambda(2M\pm a)=\lambda(a),
\]
which holds precisely when $a$ is a power of two. The left neighbours of the middle are therefore the old right-boundary positions, whose colours lie in $\{A,D,E\}$. The right neighbours are their reflected copies, whose colours lie in $\{D,A,C\}$. Neither palette contains the middle's colour $B$.
\end{enumerate}

The enlarged graph is therefore properly coloured.

It remains to check the boundary condition, so that the construction can be repeated. At the new left boundary, the special positions are the old ones together with the middle $M$; their colours still lie in $\{A,B,C\}$. At the new right boundary, the special positions are reflections of the old left-boundary positions, again together with the middle. Their colours lie in $\{D,E,B\}$.

The new boundary palettes are therefore $\{A,B,C\}$ and $\{B,D,E\}$. Crucially, these are two three-colour palettes sharing exactly one colour. The induction condition is restored, now with $B$ as the shared colour.

We can consequently repeat the construction for blocks of every required size, using the same five colours.
\end{proof}

\clearpage
\section{The coil bound}
\label{app:coil-proof}
\begin{proof}
For $m=1,2,3$ ($Q_{21},Q_{22},Q_{23}$), we prove the bound directly. Retain the endpoint tips and join the first and last Gray layers. In $Q_{21}$, use $I_{r_0}(t_0)$ and $O_{r_0}(t_0)$. In $Q_{22}$, use boundary tip $t_0$ and interior tip $t_1$. In $Q_{23}$, use boundary tip $t_0$ and interior tips $t_1,t_2,t_3$. The distinct tips prevent unwanted tip contacts. The resulting lengths, averaging the two root orders in the last two cases, give
\[
\begin{aligned}
b(21)&\ge |I_{r_0}(t_0)|+|O_{r_0}(t_0)|=743{,}042,\\
b(22)&\ge (S_0+S_1)/2=1{,}485{,}848,\\
b(23)&\ge (S_0+S_1+S_2+S_3)/2=2{,}971{,}492.
\end{aligned}
\]
Each exceeds $(17/48)2^d$ in its dimension.

For $m\ge4$, replace the boundary paths by $I_{r_0}(t_5)$ and $O_{r_1}(t_5)$, retaining $t_5$: its exclusive use at the two ends supplies the closing edge between the first and last Gray layers, while the root rule and Lemma~\ref{lem:assembly} exclude other unwanted contacts.

Across the same ten choices, the interior paths contribute $(N-1)T$ vertices in total and the boundary paths contribute $5S_5$. No vertices are deleted, and a cycle has as many edges as vertices. Therefore, with $N=2^{d-21}$,
\[
\begin{aligned}
b(d)&\ge\frac{(N-1)T+5S_5}{10}
=\frac{928503}{2621440}\,2^d-\frac{2192}{5}\\
&\ge\frac{17}{48}\,2^d+\frac{752}{15}
>\frac{17}{48}\,2^d,
\end{aligned}
\]
where the second line uses $d\ge24$.
\end{proof}

\clearpage
\section{Comparison plots}
\begin{figure}[H]
\centering
\includegraphics[width=\textwidth]{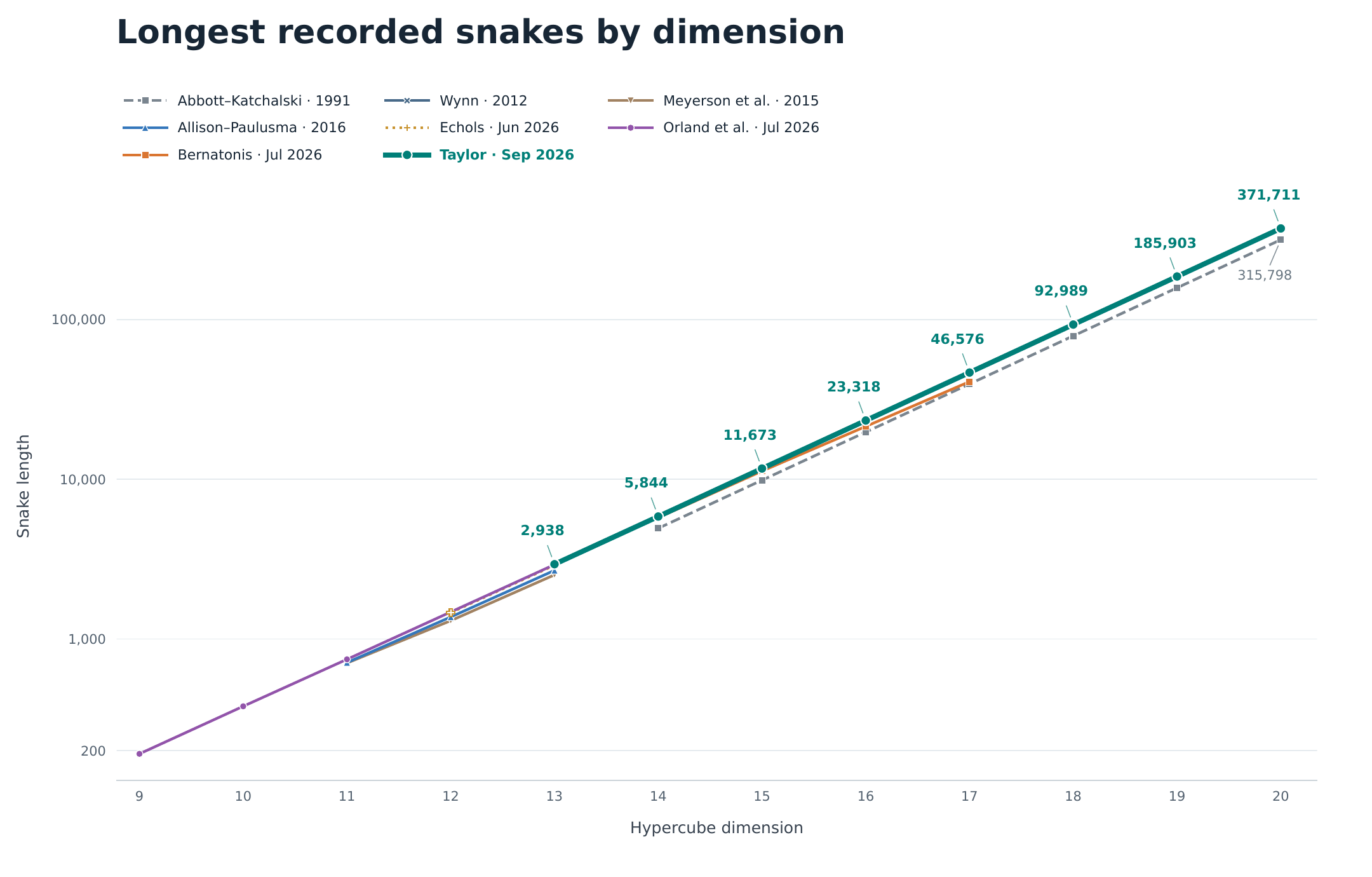}
\end{figure}

\begin{figure}[H]
\centering
\includegraphics[width=\textwidth]{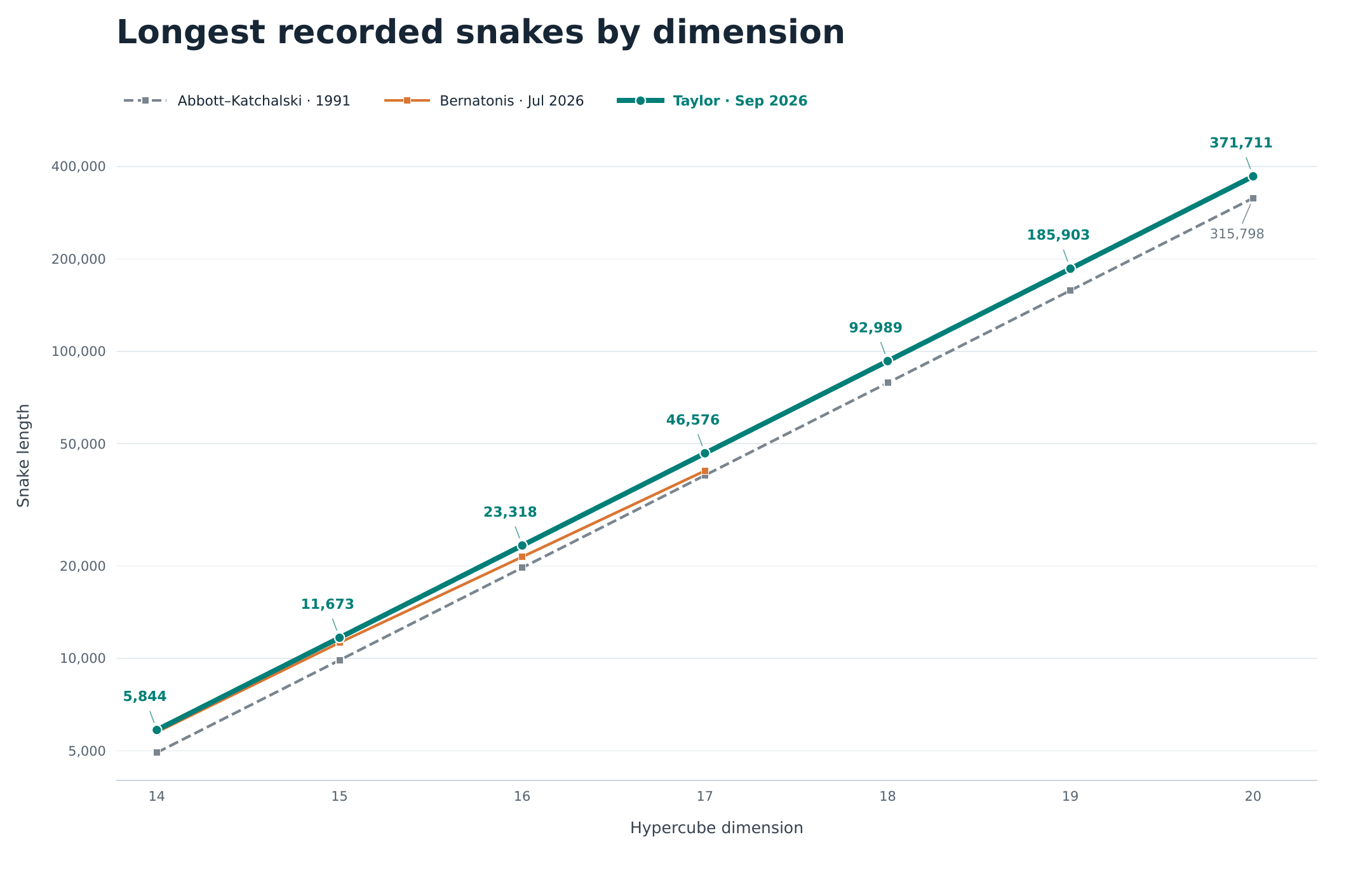}
\end{figure}

\end{document}